\documentclass[12pt]{amsart}
\usepackage{amsfonts}
\usepackage{ifthen}
\usepackage{amsthm}
\usepackage{amsmath}
\usepackage{graphicx}
\usepackage{amscd,amssymb,amsthm}
\usepackage[numbers,sort&compress]{natbib}

\newtheorem{theorem}{Theorem}
\newtheorem{lemma}{Lemma}

\DeclareMathOperator{\RM}{Re}
\begin{document}
\title[Composition formulas of Bessel-Struve Kernel function]{Composition formulas of Bessel-Struve Kernel function}
\author{K.S Nisar}
\address{Department of Mathematics, College of Arts and Science-Wadi
Addwasir,\\
Prince Sattam bin Abdulaziz University, Saudi Arabia}
\email{ksnisar1@gmail.com}
\author{S.R Mondal}
\address{Department of Mathematics and Statistics, College of Science,\\
King Faisal University, Saudi Arabia}
\email{smondal@kfu.edu.sa}
\author{P.Agarwal}
\address{Anand International college of engineering,\\
Jaipur, India}
\email{goyal.praveen2011@gmail.com}


\begin{abstract}
The generalized operators of fractional integration involving Appell's
function $F_{3}(.) $ due to Marichev-Saigo-Maeda, is applied to the Bessel
Struve kernel function $S_{\alpha }\left( \lambda z\right),\lambda ,z\in 
\mathbb{C}$ to obtain the results in terms of generalized Wright
functions.The pathway integral representations Bessel Struve kernel function
and its relation between many other functions also derived in this study.
\end{abstract}

\maketitle

\section{Introduction}

The fractional calculus is one of the most rapidly growing subject of
applied mathematics that deals with derivatives and integrals of arbitrary
orders. The applications of fractional integral operator involving various
special functions has found in various sub fields such as statistical
distribution theory, control theory, fluid dynamics, stochastic dynamical
system, plasma, image processing, nonlinear biological systems,
astrophysics, and in quantum mechanics (see \cite{Balenu-b1,
Kiriyakova-b1,Balenu-b2,Purohit-b1}).

The influence of fractional integral operators involving various special
functions in fractional calculus is very important as its significance and
applications in various sub-fields of applied mathematical analysis. Many
studies related to the fractional calculus found in the papers of Love \cite%
{Love}, McBride \cite{McBride}, Kalla \cite{Kalla1, Kalla2}, Kalla and
Saxena \cite{Kalla3, Kalla4}, Saigo \cite{Saigo2,Saigo1, Saigo3} , Saigo and
Maeda \cite{Saigo4}, Kiryakova \cite{Kiriyakova-b2} gave various
applications and extensions of hypergeometric operators of fractional
integration. A comprehensive explanation of such operators can be found in
the research monographs by Miller and Ross \cite{Miller-Ross} and Kiryakova 
\cite{Kiriyakova-b2}.

A useful generalization of the hypergeometric fractional integrals,
including the Saigo operators (\cite{Saigo1, Saigo2, Saigo3}), has been
introduced by Marichev \cite{Marichev1} (see details in Samko et al. \cite%
{Samko1}, p. 194) and later extended and studied by Saigo and Maeda (\cite%
{Saigo4}, p.393) in term of any complex order with Appell function $%
F_{3}\left( .\right) $ in the kernel, as follows:

\bigskip Let $\alpha $, $\alpha ^{\prime }$, $\beta $, $\beta ^{\prime }$, $%
\gamma \in \mathbb{C}$ and $x>0$, then the generalized fractional calculus
operator involving the Appell functions, or Horn's function are defined as
follows:

\begin{equation}
\left( I_{0,{+}}^{\alpha ,\alpha ^{\prime },\beta ,\beta ^{\prime },\gamma
}f\right) (x)=\frac{x^{-\alpha }}{\Gamma (\gamma )}\int_{0}^{x}(x-t)^{\gamma
-1}t^{-\alpha ^{\prime }}F_{3}\left( \alpha ,\alpha ^{\prime },\beta ,\beta
^{\prime };\gamma ;1-\frac{t}{x},1-\frac{x}{t}\right) f(t)dt,  \label{1}
\end{equation}

and

\begin{equation}
\left( I_{0,{-}}^{\alpha ,\alpha ^{\prime },\beta ,\beta ^{\prime },\gamma
}f\right) (x)=\frac{x^{-\alpha ^{\prime }}}{\Gamma (\gamma )}%
\int_{x}^{\infty }(t-x)^{\gamma -1}t^{-\alpha }F_{3}\left( \alpha ,\alpha
^{\prime },\beta ,\beta ^{\prime };\gamma ;1-\frac{t}{x},1-\frac{x}{t}%
\right) f(t)dt,  \label{2}
\end{equation}

with $\RM(\gamma )>0$. The generalized fractional integral operators of the
type $(\ref{1})$ and $(\ref{2})$ has been introduced by Marichev \cite%
{Marichev1} and later extended and studied by Saigo and Maeda \cite{Saigo4}%
.This operator together known as the Marichev-Saigo-Maeda operator. For the
definition of the Appell function $F_{3}\left( :\right) $ the interested
readers may refer to the monograph by Srivastava and Karlson \cite%
{Srivastava-Karlson} (see also Erd%
%TCIMACRO{\U{b4}}%
%BeginExpansion
\'{}%
%EndExpansion
elyi et al. \cite{Erdelyi} and Prudnikov et al. \cite{Prudnikov}).

The fractional integral operator has many interesting application in various
sub-fields in applicable mathematical analysis For example, \cite%
{Kim-Lee-HMSri} has applications related to certain class of complex
analytic functions. The results given in \cite{Kir1, Miller-Ross,
HMSri-frac-Hfunction, Nisar} can be referred for some basic results on
fractional calculus.

Following two results given by Saigo et. al. \cite{saxena-saigo, Saigo4} are
needed in sequel.

\begin{lemma}
\label{lem-1} Let $\alpha ,\alpha ^{\prime },\beta ,\beta ^{\prime },\gamma
,\rho \in \mathbb{C}$ be such that $\RM{(\gamma)}>0$ and 
\begin{equation*}
\RM{(\rho)}>\max \{0,\RM{(\alpha-\alpha'-\beta-\gamma)},\RM{(\alpha'-\beta')}%
\}.
\end{equation*}

Then there exists the relation

\begin{eqnarray*}
\left( I_{0,+}^{\alpha ,\alpha ^{\prime },\beta ,\beta ^{\prime },\gamma
}~t^{\rho -1}\right) (x) =\Gamma \left[ 
\begin{array}{lll}
\rho , & \rho +\gamma -\alpha -\alpha ^{\prime }-\beta, & \rho +\beta {%
^{\prime }}-\alpha {^{\prime }} \\ 
\rho +\beta {^{\prime }}, & \rho +\gamma -\alpha -\alpha {^{\prime }}, & 
\rho +\gamma -\alpha {^{\prime }}-\beta%
\end{array}%
\right] x^{\rho -\alpha -\alpha ^{\prime }+\gamma -1},
\end{eqnarray*}

where

\begin{equation*}
\Gamma \left[ 
\begin{array}{lll}
a, & b, & c \\ 
d, & e, & f%
\end{array}%
\right] =\frac{\Gamma {(a)}\Gamma {(b)}\Gamma {(c)}}{\Gamma {(d)}\Gamma {(e)}%
\Gamma {(f)}}.
\end{equation*}
\end{lemma}

\begin{lemma}
\label{lem-2} Let $\alpha,\alpha^{\prime },\beta,\beta^{\prime },\gamma,\rho
\in \mathbb{C}$ be such that $\RM{(\gamma)}>0$ and 
\begin{equation*}
\RM{(\rho)}< 1+ \min\{\RM{(-\beta)}, \RM{(\alpha+\alpha'-\gamma)}, %
\RM{(\alpha+\beta'-\gamma)}\}.
\end{equation*}

Then there exists the relation

\begin{eqnarray*}
&&\left( I_{0,-}^{\alpha ,\alpha ^{\prime },\beta ,\beta ^{\prime },\gamma
}\;t^{\rho -1}\right) (x)=\Gamma \left[ 
\begin{array}{lll}
1-\rho -\gamma +\alpha +\alpha ^{\prime },1-\rho +\alpha +\beta {^{\prime }}%
,1-\rho -\beta &  &  \\ 
1-\rho ,1-\rho +\alpha +\alpha ^{\prime }+\beta +\beta ^{\prime }-\gamma
,1-\rho +\alpha -\beta &  & 
\end{array}%
\right] x^{\rho -\alpha -\alpha ^{\prime }+\gamma -1}.
\end{eqnarray*}
\end{lemma}

\section{Representations in term of generalized Wright functions}

\label{sec:wright-fun} The generalized Wright hypergeometric function $%
{}_{p}\psi _{q}(z)$ which is defined by the series 
\begin{equation}
{}_{p}\psi _{q}(z)={}_{p}\psi _{q}\left[ 
\begin{array}{c}
(a_{i},\alpha _{i})_{1,p} \\ 
(b_{j},\beta _{j})_{1,q}%
\end{array}%
\bigg|z\right] =\displaystyle\sum_{k=0}^{\infty }\dfrac{\prod_{i=1}^{p}%
\Gamma (a_{i}+\alpha _{i}k)}{\prod_{j=1}^{q}\Gamma (b_{j}+\beta _{j}k)}%
\dfrac{z^{k}}{k!}.  \label{Wright1}
\end{equation}%
Here $a_{i},b_{j}\in \mathbb{C}$, and $\alpha _{i},\beta _{j}\in \mathbb{R}$
($i=1,2,\ldots ,p;$ $j=1,2,\ldots ,q$). Asymptotic behavior of this function
for large values of argument of $z\in {\mathbb{C}}$ were studied in \cite{
Fox} and under the condition 
\begin{equation}
\displaystyle\sum_{j=1}^{q}\beta _{j}-\displaystyle\sum_{i=1}^{p}\alpha
_{i}>-1  \label{Wright2}
\end{equation}%
in \cite{Wright-2,Wright-3}. Properties of this generalized Wright function
were investigated in \cite{Kilbas, Kilbas-itsf, Kilbas-frac}. In particular,
it was proved \cite{Kilbas} that ${}_{p}\psi _{q}(z)$, $z\in {\mathbb{C}}$
is an entire function under the condition ($\ref{Wright2}$). Interesting
results related to generalized Wright functions are also given in \cite%
{HMSri-Wright}. For proceeding the coming subsections we nee to recall the
following definitions. The Bessel and modified Bessel function of first
kind , the Struve function $H_{v}\left( z\right) $ and modified Struve
function $L_{v}\left( z\right) $ possess power series representation of the
form \cite{Watson}%
\begin{equation}
J_{\upsilon }\left( z\right) =\sum\limits_{k=0}^{\infty }\frac{\left(
-1\right) ^{k}\left( \frac{z}{2}\right) ^{2k+\upsilon }}{\Gamma \left(
k+\upsilon +1\right) k!},  \label{Bessel1}
\end{equation}

\begin{equation}
I_{\upsilon }\left( z\right) =\sum\limits_{k=0}^{\infty }\frac{\left( \frac{z%
}{2}\right) ^{2k+\upsilon }}{\Gamma \left( k+\upsilon +1\right) k!}
\label{Bessel2}
\end{equation}

\begin{equation}
H_{\upsilon }\left( z\right) =\left( \frac{z}{2}\right) ^{\upsilon
+1}\sum\limits_{k=0}^{\infty }\frac{\left( -1\right) ^{k}}{\Gamma \left( k+%
\frac{3}{2}\right) \Gamma \left( k+\upsilon +\frac{1}{2}\right) }\left( 
\frac{z}{2}\right) ^{2k}  \label{eqn-1-Struve}
\end{equation}

and

\begin{equation}
L_{\upsilon }\left( z\right) =\left( \frac{z}{2}\right) ^{\upsilon
+1}\sum\limits_{k=0}^{\infty }\frac{1}{\Gamma \left( k+\frac{3}{2}\right)
\Gamma \left( k+\upsilon +\frac{1}{2}\right) }\left( \frac{z}{2}\right) ^{2k}
\label{eqn-2-Struve}
\end{equation}

The purpose of this work is to investigate compositions of integral
transforms ($\ref{1}$) and ($\ref{2}$) with the Bessel Struve kernel
function $S_{\alpha }\left( \lambda z\right) $ defined for $z,\lambda \in 
\mathbb{C}$. The Bessel-Struve kernel $S_{\alpha }\left( \lambda z\right)
,\lambda \in \mathbb{C},$ \cite{Gasmi} which is unique solution of the
initial value problem $l_{\alpha }u\left( z\right) =\lambda ^{2}u\left(
z\right) $ with the initial conditions $u\left( 0\right) =1$ and $%
u^{^{\prime }}\left( 0\right) =\lambda \Gamma \left( \alpha +1\right) / (%
\sqrt{\pi }\Gamma \left( \alpha +3/2\right)) $ is given by $S_{\alpha
}\left( \lambda z\right) =j_{\alpha }\left( i\lambda z\right) -ih_{\alpha
}\left( i\lambda z\right) ,\forall z\in C $ where $j_{\alpha }$ and $%
h_{\alpha }$ are the normalized Bessel and Struve functions.

Moreover, the Bessel-Struve kernel is a holomorphic function on $\mathbb{C}%
\times \mathbb{C}$ and it can be expanded in a power series in the form 
\begin{equation}
S_{\alpha }\left( \lambda z\right) =\sum_{n=0}^{\infty }\frac{\left( \lambda
z\right) ^{n}\Gamma \left( \alpha +1\right) \Gamma \left( \frac{n+1}{2}%
\right) }{\sqrt{\pi }n!\Gamma \left( \frac{n}{2}+\alpha +1\right) },
\label{3}
\end{equation}

Now, we derive the following theorems,

\begin{theorem}
\label{thm-1} Let $\alpha ,\alpha ^{\prime },\beta ,\beta ^{\prime },\gamma
,\rho ,p,b,c,\lambda ,\nu \in \mathbb{C}$ . Suppose that $\RM{(\gamma)}>0$
and $\RM{(\rho+ n)}>\max \{0,\RM{(\alpha+\alpha'+\beta-\gamma)},%
\RM{(\alpha'-\beta')}\}.$ Then 
\begin{align*}
& \left( I_{0,+}^{\alpha ,\alpha ^{\prime },\beta ,\beta ^{\prime },\gamma
}t^{\rho -1}S_{\nu }\left( \lambda t\right) \right) (x)=x^{\rho +\gamma
-\alpha -\alpha ^{^{\prime }}-1}\frac{\Gamma \left( \nu +1\right) }{\sqrt{%
\pi }} \\
& \times _{4}\Psi _{4}\left[ 
\begin{array}{c}
\left( \frac{1}{2},\frac{1}{2}\right) ,\left( \rho ,1\right) ,\left( \rho
+\gamma -\alpha -\alpha ^{^{\prime }}-\beta ,1\right) ,\left( \rho +\beta
^{^{\prime }}-\alpha ^{^{\prime }},1\right) \\ 
\left( \nu +1,\frac{1}{2}\right) ,\left( \rho +\beta ^{^{\prime }},1\right)
,\left( \rho +\gamma -\alpha -\alpha ^{^{\prime }},1\right) ,\left( \rho
+\gamma -\alpha ^{^{\prime }}-\beta ,n\right)%
\end{array}%
\left|\right. x\lambda \right]
\end{align*}
\end{theorem}

\begin{proof}
Using $(\ref{1})$ and $(\ref{3})$,

\begin{eqnarray*}
&&\left( I_{0,+}^{\alpha ,\alpha ^{\prime },\beta ,\beta ^{\prime },\gamma
}t^{\rho -1}S_{\nu }\left( \lambda t\right) \right) (x) \\
&=&\left( I_{0,+}^{\alpha ,\alpha ^{\prime },\beta ,\beta ^{\prime },\gamma
}\sum_{n=0}^{\infty }\frac{\left( \lambda \right) ^{n}\Gamma \left( \nu
+1\right) \Gamma \left( \left( n+1\right) /2\right) }{\sqrt{\pi }n!\Gamma
\left( n/2+\nu +1\right) }t^{\rho +n-1}\right) (x)
\end{eqnarray*}

By changing the order of integration and summation, 
\begin{eqnarray*}
&&\left( I_{0,+}^{\alpha ,\alpha ^{\prime },\beta ,\beta ^{\prime },\gamma
}t^{\rho -1}S_{\nu }\left( \lambda t\right) \right) (x) \\
&=&\sum_{n=0}^{\infty }\frac{\left( \lambda \right) ^{n}\Gamma \left( \nu
+1\right) \Gamma \left( \left( n+1\right) /2\right) }{\sqrt{\pi }n!\Gamma
\left( n/2+\nu +1\right) }(I_{0,+}^{\alpha ,\alpha ^{\prime },\beta ,\beta
^{\prime },\gamma }t^{\rho +n-1})(x).
\end{eqnarray*}%
Hence replacing $\rho $ by $\rho +n$ in Lemma $\ref{lem-1}$ , we obtain

\begin{eqnarray*}
&&\left( I_{0,+}^{\alpha ,\alpha ^{\prime },\beta ,\beta ^{\prime },\gamma
}t^{\rho -1}S_{\nu }\left( \lambda t\right) \right) (x) \\
&=&\sum_{n=0}^{\infty }\frac{\Gamma \left( \nu +1\right) \Gamma \left( \frac{%
n+1}{2}\right) \lambda ^{n}}{\sqrt{\pi }n!\Gamma \left( n/2+\nu +1\right) }
\\
&&\times \frac{\Gamma \left( \rho +n\right) \Gamma \left( \rho +n+\gamma
-\alpha -\alpha ^{^{\prime }}-\beta \right) \Gamma \left( \rho +n+\beta
^{^{\prime }}-\alpha ^{^{\prime }}\right) }{\Gamma \left( \rho +n+\beta
^{^{\prime }}\right) \Gamma \left( \rho +n+\gamma -\alpha -\alpha ^{^{\prime
}}\right) \Gamma \left( \rho +n+\gamma -\alpha ^{^{\prime }}-\beta \right) }%
x^{\rho +n+\gamma -\alpha -\alpha ^{^{\prime }}-1}
\end{eqnarray*}%
\begin{eqnarray*}
&=&\sum_{n=0}^{\infty }\frac{\Gamma \left( \nu +1\right) \Gamma \left( \frac{%
n}{2}-\frac{1}{2}+1\right) \lambda ^{n}}{\sqrt{\pi }n!\Gamma \left( n/2+\nu
+1\right) } \\
&&\times \frac{\Gamma \left( \rho +n\right) \Gamma \left( \rho +n+\gamma
-\alpha -\alpha ^{^{\prime }}-\beta \right) \Gamma \left( \rho +n+\beta
^{^{\prime }}-\alpha ^{^{\prime }}\right) }{\Gamma \left( \rho +n+\beta
^{^{\prime }}\right) \Gamma \left( \rho +n+\gamma -\alpha -\alpha ^{^{\prime
}}\right) \Gamma \left( \rho +n+\gamma -\alpha ^{^{\prime }}-\beta \right) }%
x^{\rho +n+\gamma -\alpha -\alpha ^{^{\prime }}-1}
\end{eqnarray*}%
\begin{eqnarray*}
&=&\frac{x^{\rho +\gamma -\alpha -\alpha ^{^{\prime }}-1}}{\sqrt{\pi }}%
\Gamma \left( \nu +1\right) \sum_{n=0}^{\infty }\frac{\Gamma \left( \frac{n}{%
2}+\frac{1}{2}\right) }{\Gamma \left( n/2+\nu +1\right) } \\
&&\times \frac{\Gamma \left( \rho +n\right) \Gamma \left( \rho +n+\gamma
-\alpha -\alpha ^{^{\prime }}-\beta \right) \Gamma \left( \rho +n+\beta
^{^{\prime }}-\alpha ^{^{\prime }}\right) }{\Gamma \left( \rho +n+\beta
^{^{\prime }}\right) \Gamma \left( \rho +n+\gamma -\alpha -\alpha ^{^{\prime
}}\right) \Gamma \left( \rho +n+\gamma -\alpha ^{^{\prime }}-\beta \right) }%
\frac{\left( x\lambda \right) ^{n}}{n!}
\end{eqnarray*}%
\begin{eqnarray*}
&=&\frac{x^{\rho +\gamma -\alpha -\alpha ^{^{\prime }}-1}}{\sqrt{\pi }}%
\Gamma \left( \nu +1\right) \\
&&\times _{4}\Psi _{4}\left[ 
\begin{array}{c}
\left( \frac{1}{2},\frac{1}{2}\right) ,\left( \rho ,1\right) ,\left( \rho
+\gamma -\alpha -\alpha ^{^{\prime }}-\beta ,1\right) ,\left( \rho +\beta
^{^{\prime }}-\alpha ^{^{\prime }},1\right) \\ 
\left( \nu +1,\frac{1}{2}\right) ,\left( \rho +\beta ^{^{\prime }},1\right)
,\left( \rho +\gamma -\alpha -\alpha ^{^{\prime }},1\right) ,\left( \rho
+\gamma -\alpha ^{^{\prime }}-\beta ,n\right)%
\end{array}%
\left| \lambda x\right. \right]
\end{eqnarray*}
\end{proof}

\begin{theorem}
\label{thm-2} Let $\alpha ,\alpha ^{\prime },\beta ,\beta ^{\prime },\gamma
,\rho ,p,b,c,\lambda ,\nu \in \mathbb{C}$ . Suppose that $\RM{(\gamma)}>0$
and $\RM{(\rho-
n)}<1+\min \{\RM{(-\beta)},\RM{(\alpha+\alpha'-\gamma)},\RM{(\alpha+\beta'-%
\gamma)}\}.$ Then 
\begin{align*}
& \left( I_{0,-}^{\alpha ,\alpha ^{\prime },\beta ,\beta ^{\prime },\gamma
}t^{\rho -1}S_{\nu }\left( \frac{\lambda }{t}\right) \right) (x) \\
& \quad =\frac{x^{\rho -\alpha -\alpha ^{^{\prime }}+\gamma -1}}{\sqrt{\pi }}%
\Gamma \left( \nu +1\right) \\
& \times _{4}\Psi _{4}\left[ 
\begin{array}{c}
\left( \frac{1}{2},\frac{1}{2}\right) ,\left( 1-\rho -\gamma +\alpha +\alpha
^{^{\prime }},1\right) ,\left( 1-\rho +\alpha +\beta ^{^{\prime }},1\right)
,\left( 1-\rho -\beta ^{^{\prime }},1\right) \\ 
\left( \nu +1,\frac{1}{2}\right) ,\left( 1-\rho ,1\right) ,\left( 1-\rho
+\alpha +\alpha ^{^{\prime }}+\beta +\beta ^{^{\prime }}-\gamma ,1\right)
,\left( 1-\rho +n+\alpha +\beta ,1\right)%
\end{array}%
\left\vert \lambda x\right. \right]
\end{align*}
\end{theorem}

\begin{proof}
Using $(\ref{2})$ and $(\ref{3})$, and then changing the order of
integration and summation, 
\begin{eqnarray*}
&&\left( I_{0,-}^{\alpha ,\alpha {^{\prime }},\beta ,\beta ^{\prime },\gamma
}t^{\rho -1}S_{\nu }\left( \frac{\lambda }{t}\right) \right) (x)
=\sum_{n=0}^{\infty }\frac{\left( \lambda \right) ^{n}\Gamma \left( \nu
+1\right) \Gamma \left( \left( n+1\right) /2\right) }{\sqrt{\pi }n!\Gamma
\left( n/2+\nu +1\right) }\left( I_{0,-}^{\alpha ,\alpha ^{\prime },\beta
,\beta ^{\prime },\gamma }t^{\rho -n-1}\right) \left( x\right) .
\end{eqnarray*}%
using Lemm $\ref{lem-2}$, we obtain%
\begin{eqnarray*}
&=&\sum_{n=0}^{\infty }\frac{\left( \lambda \right) ^{n}\Gamma \left( \nu
+1\right) \Gamma \left( \left( n+1\right) /2\right) }{\sqrt{\pi }n!\Gamma
\left( n/2+\nu +1\right) } \\
&&\times \frac{\Gamma \left( 1-\rho -n-\gamma +\alpha +\alpha ^{^{\prime
}}\right) \Gamma \left( 1-\rho -n+\alpha +\beta ^{^{\prime }}\right) \Gamma
\left( 1-\rho -n-\beta \right) }{\Gamma \left( 1-\rho -n\right) \Gamma
\left( 1-\rho -n+\alpha +\alpha ^{^{\prime }}+\beta +\beta ^{^{\prime
}}-\gamma \right) \Gamma \left( 1-\rho -n+\alpha +\beta \right) }x^{\rho
-n-\alpha -\alpha ^{^{\prime }}+\gamma -1}
\end{eqnarray*}%
\begin{eqnarray*}
&=&\frac{x^{\rho -\alpha -\alpha ^{^{\prime }}+\gamma -1}}{\sqrt{\pi }}\frac{%
\Gamma \left( \nu +1\right) \Gamma \left( \left( n+1\right) /2\right) }{%
\Gamma \left( n/2+\nu +1\right) } \\
&&\times \frac{\Gamma \left( 1-\rho +n-\gamma +\alpha +\alpha ^{^{\prime
}}\right) \Gamma \left( 1-\rho +n+\alpha +\beta ^{^{\prime }}\right) \Gamma
\left( 1-\rho +n-\beta \right) }{\Gamma \left( 1-\rho +n\right) \Gamma
\left( 1-\rho +n+\alpha +\alpha ^{^{\prime }}+\beta +\beta ^{^{\prime
}}-\gamma \right) \Gamma \left( 1-\rho +n+\alpha +\beta \right) }\frac{%
\left( x\lambda \right) ^{n}}{n!}
\end{eqnarray*}%
\begin{eqnarray*}
&=&\frac{x^{\rho -\alpha -\alpha ^{^{\prime }}+\gamma -1}}{\sqrt{\pi }}%
\Gamma \left( \nu +1\right) \\
&&\times _{4}\Psi _{4}\left[ 
\begin{array}{c}
\left( \frac{1}{2},\frac{1}{2}\right) ,\left( 1-\rho -\gamma +\alpha +\alpha
^{^{\prime }},1\right) ,\left( 1-\rho +\alpha +\beta ^{^{\prime }},1\right)
,\left( 1-\rho -\beta ^{^{\prime }},1\right) \\ 
\left( \nu +1,\frac{1}{2}\right) ,\left( 1-\rho ,1\right) ,\left( 1-\rho
+\alpha +\alpha ^{^{\prime }}+\beta +\beta ^{^{\prime }}-\gamma ,1\right)
,\left( 1-\rho +n+\alpha +\beta ,1\right)%
\end{array}%
\left\vert x\lambda \right. \right]
\end{eqnarray*}
\end{proof}

\subsection{Representation of Bessel Struve kernel function interms of
exponential function}

In this subsection we represent the Bessel Struve function interms of
exponential function. Also, we derive the Marichev Saigo Maeda operator
representation of special cases. The representation Bessel Struve Kernel
function interms of exponential function as:%
\begin{equation}
S_{\frac{-1}{2}}\left( x\right) =e^{x},  \label{e1}
\end{equation}

\begin{equation}
S_{\frac{1}{2}}\left( x\right) =\frac{-1+e^{x}}{x}.  \label{e2}
\end{equation}

Now, we give the the following theorems

\begin{theorem}
Let $\alpha ,\alpha ^{\prime },\beta ,\beta ^{\prime },\gamma ,\rho
,p,b,c,\lambda \in \mathbb{C}$ . Suppose that $\RM{(\gamma)}>0$ and $%
\RM{(\rho+ n)}>\max \{0,\RM{(\alpha+\alpha'+\beta-\gamma)},%
\RM{(\alpha'-\beta')}\}.$ Then

\begin{align*}
& \left( I_{0,+}^{\alpha ,\alpha ^{\prime },\beta ,\beta ^{\prime },\gamma
}t^{\rho -1}e^{t}\right) (x)=x^{\rho -\alpha -\alpha ^{^{\prime }}+\gamma -1}
\\
& \times _{3}\Psi _{3}\left[ 
\begin{array}{c}
\left( \rho ,1\right) ,\left( \rho +\gamma -\alpha -\alpha ^{^{\prime
}}-\beta ,1\right) ,\left( \rho +\beta ^{^{\prime }}-\alpha ^{^{\prime
}},1\right) ,\left( 1-\rho -\beta ^{^{\prime }},1\right) \\ 
\left( \rho +\beta ^{^{\prime }},1\right) ,\left( \rho +\gamma -\alpha
-\alpha ^{^{\prime }},1\right) ,\left( \rho +\gamma -\alpha ^{^{\prime
}}-\beta ,1\right)%
\end{array}%
\left\vert x\right. \right]
\end{align*}
\end{theorem}

\begin{proof}
From equation 1 and the definition of Bessel Struve kernel function, we have%
\begin{align*}
\left( I_{0,+}^{\alpha ,\alpha ^{\prime },\beta ,\beta ^{\prime },\gamma
}t^{\rho -1}e^{t}\right) (x)& = \left( I_{0,+}^{\alpha ,\alpha ^{\prime
},\beta ,\beta ^{\prime },\gamma }t^{\rho -1}\sum\limits_{n=0}^{\infty }%
\frac{\Gamma \left( -1/2+1\right) \Gamma \left( \frac{n+1}{2}\right) }{\sqrt{%
\pi }n!\Gamma \left( n/2-1/2+1\right) }t^{n}\right) \left( x\right) \\
&=\sum\limits_{n=0}^{\infty }\frac{\Gamma \left( 1/2\right) \Gamma \left( 
\frac{n+1}{2}\right) }{\sqrt{\pi }n!\Gamma \left( \frac{n+1}{2}\right) }%
\left( I_{0,+}^{\alpha ,\alpha ^{\prime },\beta ,\beta ^{\prime },\gamma
}t^{\rho +n-1}\right) \left( x\right)
\end{align*}
This together with Lemma $\ref{lem-1}$ yields 
\begin{align*}
\left( I_{0,+}^{\alpha ,\alpha ^{\prime },\beta ,\beta ^{\prime },\gamma
}t^{\rho -1}e^{t}\right) (x)& =x^{\rho -\alpha -\alpha ^{^{\prime }}+\gamma
-1}\sum\limits_{n=0}^{\infty }\frac{x^{n}}{n!}\Gamma \left[ 
\begin{array}{c}
\rho +n,\rho +n+\gamma -\alpha -\alpha ^{^{\prime }}-\beta ,\rho +n+\beta
^{^{\prime }}-\alpha ^{^{\prime }} \\ 
\rho +n+\beta ^{^{\prime }},\rho +n+\gamma -\alpha -\alpha ^{^{\prime
}},\rho +n+\gamma -\alpha ^{^{\prime }}-\beta%
\end{array}%
\right]
\end{align*}
which is the desired result.
\end{proof}

\begin{theorem}
Let $\alpha ,\alpha ^{\prime },\beta ,\beta ^{\prime },\gamma ,\rho
,p,b,c,\lambda \in \mathbb{C}$ . Suppose that $\RM{(\gamma)}>0$ and $%
\RM{(\rho+ n)}>\max \{0,\RM{(\alpha+\alpha'+\beta-\gamma)},%
\RM{(\alpha'-\beta')}\}.$ Then

\begin{align*}
& \left( I_{0,+}^{\alpha ,\alpha ^{\prime },\beta ,\beta ^{\prime },\gamma
}t^{\rho -1}\left( \frac{-1+e^{t}}{t}\right) \right) (x)\quad =x^{\rho
-\alpha -\alpha ^{^{\prime }}+\gamma -1} \\
& \times _{4}\Psi _{4}\left[ 
\begin{array}{c}
\left( \frac{1}{2},\frac{1}{2}\right) ,\left( \rho ,1\right) ,\left( \rho
+\gamma -\alpha -\alpha ^{^{\prime }}-\beta ,1\right) ,\left( \rho +\beta
^{^{\prime }}-\alpha ^{^{\prime }},1\right) ; \\ 
\left( \frac{1}{2},\frac{3}{2}\right) ,\left( \rho +\beta ^{^{\prime
}},1\right) ,\left( \rho +\gamma -\alpha -\alpha ^{^{\prime }},1\right)
,\left( \rho +\gamma -\alpha ^{^{\prime }}-\beta ,1\right)%
\end{array}%
\left\vert x\right. \right]
\end{align*}
\end{theorem}

\begin{proof}
From equation 1 and the definition of Bessel Struve kernel function, we have%
\begin{align*}
\left( I_{0,+}^{\alpha ,\alpha ^{\prime },\beta ,\beta ^{\prime },\gamma
}t^{\rho -1}\left( \frac{-1+e^{t}}{t}\right) \right) (x) &=\left(
I_{0,+}^{\alpha ,\alpha ^{\prime },\beta ,\beta ^{\prime },\gamma }t^{\rho
-1}\sum\limits_{n=0}^{\infty }\frac{\Gamma \left( 1/2+1\right) \Gamma \left( 
\frac{n+1}{2}\right) }{\sqrt{\pi }n!\Gamma \left( n/2+1/2+1\right) }%
t^{n}\right) \left( x\right) \\
&=\sum\limits_{n=0}^{\infty }\frac{\Gamma \left( 3/2\right) \Gamma \left( 
\frac{n+1}{2}\right) }{\sqrt{\pi }n!\Gamma \left( \frac{n+3}{2}\right) }%
\left( I_{0,+}^{\alpha ,\alpha ^{\prime },\beta ,\beta ^{\prime },\gamma
}t^{\rho +n-1}\right) \left( x\right)
\end{align*}
Again using Lemma $\ref{lem-1}$, we can conclude that 
\begin{align*}
&\left( I_{0,+}^{\alpha ,\alpha ^{\prime },\beta ,\beta ^{\prime },\gamma
}t^{\rho -1}\left( \frac{-1+e^{t}}{t}\right) \right) (x) \\
&=\sum\limits_{n=0}^{\infty }\frac{\Gamma \left( \frac{n+1}{2}\right) }{%
\Gamma \left( \frac{n+3}{2}\right) }\Gamma \left[ 
\begin{array}{c}
\rho +n,\rho +n+\gamma -\alpha -\alpha ^{^{\prime }}-\beta ,\rho +n+\beta
^{^{\prime }}-\alpha ^{^{\prime }} \\ 
\rho +n+\beta ^{^{\prime }},\rho +n+\gamma -\alpha -\alpha ^{^{\prime
}},\rho +n+\gamma -\alpha ^{^{\prime }}-\beta%
\end{array}%
\right]\frac{x^{x^{\rho +n-\alpha -\alpha ^{^{\prime }}+\gamma -1}}}{2n!} \\
&=\frac{x^{\rho -\alpha -\alpha ^{^{\prime }}+\gamma -1}}{2}\times
\sum\limits_{n=0}^{\infty }\frac{\Gamma \left( \frac{n+1}{2}\right) }{\Gamma
\left( \frac{n+3}{2}\right) }\Gamma \left[ 
\begin{array}{c}
\rho +n,\rho +n+\gamma -\alpha -\alpha ^{^{\prime }}-\beta ,\rho +n+\beta
^{^{\prime }}-\alpha ^{^{\prime }} \\ 
\rho +n+\beta ^{^{\prime }},\rho +n+\gamma -\alpha -\alpha ^{^{\prime
}},\rho +n+\gamma -\alpha ^{^{\prime }}-\beta%
\end{array}%
\right] \frac{x^{n}}{n!}
\end{align*}
hence the conclusion.
\end{proof}

\subsection{Relation between Bessel Struve kernel function and Bessel and
Struve function of \ first kind}

In this subsection we show the relation between $S_{\alpha }\left( x\right) $
and Bessel function $I_{v}\left( x\right) $ and Struve function $L_{v}\left(
x\right) $ by choosing particular values of $\alpha $

\begin{equation}
S_{0}\left( x\right) =I_{0}\left( x\right) +L_{0}\left( x\right) ,
\label{eqn-r1}
\end{equation}

\begin{equation}
S_{1}\left( x\right) =\frac{2I_{1}\left( x\right) +L_{1}\left( x\right) }{x},
\label{eqn-r2}
\end{equation}

then we derive the Marichev Saigo Maeda operator representation of special
cases

\begin{theorem}
Let $\alpha ,\alpha ^{\prime },\beta ,\beta ^{\prime },\gamma ,\rho
,p,b,c,\lambda \in \mathbb{C}$ . Suppose that $\RM{(\gamma)}>0$ and $%
\RM{(\rho+ n)}>\max \{0,\RM{(\alpha+\alpha'+\beta-\gamma)},%
\RM{(\alpha'-\beta')}\}.$ Then

\begin{align*}
& \left( I_{0,+}^{\alpha ,\alpha ^{\prime },\beta ,\beta ^{\prime },\gamma
}t^{\rho -1}\left( I_{0}\left( t\right) +L_{0}\left( t\right) \right)
\right) (x) \\
&=\frac{x^{\rho -\alpha -\alpha ^{^{\prime }}+\gamma -1}}{\sqrt{\pi }}
\times _{4}\Psi _{4}\left[ 
\begin{array}{c}
\left( \frac{1}{2},\frac{1}{2}\right) ,\left( \rho ,1\right) ,\left( \rho
+\gamma -\alpha -\alpha ^{^{\prime }}-\beta ,1\right) ,\left( \rho +\beta
^{^{\prime }}-\alpha ^{^{\prime }},1\right) ; \\ 
\left( \frac{1}{2},1\right) ,\left( \rho +\beta ^{^{\prime }},1\right)
,\left( \rho +\gamma -\alpha -\alpha ^{^{\prime }},1\right) ,\left( \rho
+\gamma -\alpha ^{^{\prime }}-\beta ,1\right)%
\end{array}%
\left\vert x\right. \right]
\end{align*}
\end{theorem}

\begin{proof}
From equation \ref{1} and the definition of Bessel Struve kernel function,
we have%
\begin{eqnarray*}
\left( I_{0,+}^{\alpha ,\alpha ^{\prime },\beta ,\beta ^{\prime },\gamma
}t_{0}^{\rho -1}\left( I_{0}\left( t\right) +L_{0}\left( t\right) \right)
\right) (x) &=&\left( I_{0,+}^{\alpha ,\alpha ^{\prime },\beta ,\beta
^{\prime },\gamma }t^{\rho -1}\sum\limits_{n=0}^{\infty }\frac{\Gamma \left(
1\right) \Gamma \left( \frac{n+1}{2}\right) }{\sqrt{\pi }n!\Gamma \left(
n/2+1\right) }t^{n}\right) \left( x\right) \\
&=&\sum\limits_{n=0}^{\infty }\frac{\Gamma \left( \frac{n+1}{2}\right) }{%
\sqrt{\pi }n!\Gamma \left( \frac{n}{2}+1\right) }\left( I_{0,+}^{\alpha
,\alpha ^{\prime },\beta ,\beta ^{\prime },\gamma }t^{\rho +n-1}\right)
\left( x\right)
\end{eqnarray*}

Using Lemma $\ref{lem-1}$, we obtain

\begin{align*}
&\left( I_{0,+}^{\alpha ,\alpha ^{\prime },\beta ,\beta ^{\prime },\gamma
}t_{0}^{\rho -1}\left( I_{0}\left( t\right) +L_{0}\left( t\right) \right)
\right) (x) \\
&=\sum\limits_{n=0}^{\infty }\frac{\Gamma \left( \frac{n+1}{2}\right)
x^{\rho +n-\alpha -\alpha ^{^{\prime }}+\gamma -1}}{\sqrt{\pi }n! \Gamma
\left( \frac{n}{2}+1\right) }\Gamma \left[ 
\begin{array}{c}
\rho +n,\rho +n+\gamma -\alpha -\alpha ^{^{\prime }}-\beta ,\rho +n+\beta
^{^{\prime }}-\alpha ^{^{\prime }} \\ 
\rho +n+\beta ^{^{\prime }},\rho +n+\gamma -\alpha -\alpha ^{^{\prime
}},\rho +n+\gamma -\alpha ^{^{\prime }}-\beta%
\end{array}%
\right] \\
&=\frac{x^{\rho -\alpha -\alpha ^{^{\prime }}+\gamma -1}}{\sqrt{\pi }}%
\sum\limits_{n=0}^{\infty }\frac{\Gamma \left( \frac{n+1}{2}\right) }{\Gamma
\left( \frac{n}{2}+1\right) }\frac{x^{n}}{n!}\Gamma \left[ 
\begin{array}{c}
\rho +n,\rho +n+\gamma -\alpha -\alpha ^{^{\prime }}-\beta ,\rho +n+\beta
^{^{\prime }}-\alpha ^{^{\prime }} \\ 
\rho +n+\beta ^{^{\prime }},\rho +n+\gamma -\alpha -\alpha ^{^{\prime
}},\rho +n+\gamma -\alpha ^{^{\prime }}-\beta%
\end{array}%
\right]
\end{align*}
which is the desired result.
\end{proof}

\begin{theorem}
Let $\alpha ,\alpha ^{\prime },\beta ,\beta ^{\prime },\gamma ,\rho
,p,b,c,\lambda \in \mathbb{C}$ . Suppose that $\RM{(\gamma)}>0$ and $%
\RM{(\rho+ n)}>\max \{0,\RM{(\alpha+\alpha'+\beta-\gamma)},%
\RM{(\alpha'-\beta')}\}.$ Then

\begin{align*}
& \left( I_{0,+}^{\alpha ,\alpha ^{\prime },\beta ,\beta ^{\prime },\gamma
}t^{\rho -1}\frac{\left( 2I_{1}\left( t\right) +L_{1}\left( t\right) \right)}{t}
\right) (x) \\
&=\frac{x^{\rho -\alpha -\alpha ^{^{\prime }}+\gamma -1}}{\sqrt{\pi }}
\times _{4}\Psi _{4}\left[ 
\begin{array}{c}
\left( \frac{1}{2},\frac{1}{2}\right) ,\left( \rho ,1\right) ,\left( \rho
+\gamma -\alpha -\alpha ^{^{\prime }}-\beta ,1\right) ,\left( \rho +\beta
^{^{\prime }}-\alpha ^{^{\prime }},1\right) ; \\ 
\left( \frac{1}{2},1\right) ,\left( \rho +\beta ^{^{\prime }},1\right)
,\left( \rho +\gamma -\alpha -\alpha ^{^{\prime }},1\right) ,\left( \rho
+\gamma -\alpha ^{^{\prime }}-\beta ,1\right)%
\end{array}%
\left\vert x\right. \right]
\end{align*}
\end{theorem}

\begin{proof}
From \eqref{1} and the definition of Bessel Struve kernel function, we have 
\begin{align*}
\left( I_{0,+}^{\alpha ,\alpha ^{\prime },\beta ,\beta ^{\prime },\gamma
}t^{\rho -1}\left( \frac{2I_{1}\left( t\right) +L_{1}\left( t\right) }{t}%
\right) \right) (x)& =\left( I_{0,+}^{\alpha ,\alpha ^{\prime },\beta ,\beta
^{\prime },\gamma }t^{\rho -1}\sum\limits_{n=0}^{\infty }\frac{\Gamma \left(
2\right) \Gamma \left( \frac{n+1}{2}\right) }{\sqrt{\pi }n!\Gamma \left(
n/2+2\right) }t^{n}\right) \left( x\right) \\
&=\frac{\Gamma \left( 2\right) }{\sqrt{\pi }}\sum\limits_{n=0}^{\infty }%
\frac{\Gamma \left( \frac{n+1}{2}\right) }{n!\Gamma \left( \frac{n}{2}%
+2\right) }\left( I_{0,+}^{\alpha ,\alpha ^{\prime },\beta ,\beta ^{\prime
},\gamma }t^{\rho +n-1}\right) \left( x\right).
\end{align*}
Using Lemma $\ref{lem-1}$, we obtain

\begin{eqnarray*}
&=&\frac{\Gamma \left( 2\right) }{\sqrt{\pi }}\sum\limits_{n=0}^{\infty }%
\frac{\Gamma \left( \frac{n+1}{2}\right) }{n!\Gamma \left( \frac{n}{2}%
+1\right) }\Gamma \left[ 
\begin{array}{c}
\rho +n,\rho +n+\gamma -\alpha -\alpha ^{^{\prime }}-\beta ,\rho +n+\beta
^{^{\prime }}-\alpha ^{^{\prime }} \\ 
\rho +n+\beta ^{^{\prime }},\rho +n+\gamma -\alpha -\alpha ^{^{\prime
}},\rho +n+\gamma -\alpha ^{^{\prime }}-\beta%
\end{array}%
\right] \\
&&\times \frac{x^{x^{\rho +n-\alpha -\alpha ^{^{\prime }}+\gamma -1}}}{n!}
\end{eqnarray*}%
\begin{equation*}
=\frac{x^{\rho -\alpha -\alpha ^{^{\prime }}+\gamma -1}}{\sqrt{\pi }}%
\sum\limits_{n=0}^{\infty }\frac{\Gamma \left( \frac{n+1}{2}\right) }{\Gamma
\left( \frac{n}{2}+1\right) }\frac{x^{n}}{n!}\Gamma \left[ 
\begin{array}{c}
\rho +n,\rho +n+\gamma -\alpha -\alpha ^{^{\prime }}-\beta ,\rho +n+\beta
^{^{\prime }}-\alpha ^{^{\prime }} \\ 
\rho +n+\beta ^{^{\prime }},\rho +n+\gamma -\alpha -\alpha ^{^{\prime
}},\rho +n+\gamma -\alpha ^{^{\prime }}-\beta%
\end{array}%
\right]
\end{equation*}

which is the desired result.
\end{proof}

\section{Pathway Fractional Integration of Bessel Struve Kernel function}

By considering the idea of Mathai \cite{Mathai-pathway}, Nair \cite{Nair-1}%
,introduced a pathway fractional integral operator and developed further by
Mathai and Haubold \cite{Mathai-Habold-1} ,\cite{Mathai-Habold-2} is defined
as follows:

Let $\ f\left( x\right) \in L\left( a,b\right) ,\eta \in C,R\left( \eta
\right) >0,a>0$ and the pathway parameter $\alpha <1$ as \cite%
{Praveen-pathway}, then

\begin{equation}
\left( P_{0+}^{\left( \eta ,\alpha \right) }f\right) \left( x\right)
=x^{\eta }\int\limits_{0}^{\left[ \frac{x}{a\left( 1-\alpha \right) }\right]
}1-\left[ \frac{a\left( 1-\alpha \right) t}{x}\right] ^{\frac{\eta }{\left(
1-\alpha \right) }}f\left( t\right) dt  \label{eqn-path-1}
\end{equation}

For a real scalar $\alpha$, the pathway model for scalar random variables is
represented by the following probability density function (p.d.f.):

\begin{equation}
f\left( x\right) =c\left\vert x\right\vert ^{\gamma -1}\left[ 1-a\left(
1-\alpha \right) \left\vert x\right\vert ^{\delta }\right] ^{\frac{\beta }{%
\left( 1-\alpha \right) }}  \label{eqn-path-2}
\end{equation}

provided that $-\infty <x<\infty ,\delta >0,\beta \geq 0,\left[ 1-a\left(
1-\alpha \right) \left\vert x\right\vert ^{\delta }\right] >0,$ and $\gamma
>0$ where c is the normalizing constant and $\alpha $ is called the pathway
parameter. For real $\alpha $ normalizing constant as follows: 
\begin{align*}
c =\left\{ 
\begin{array}{ccc}
\frac{1}{2}\frac{\delta \left[ a\left( 1-\alpha \right) \right] ^{\frac{%
\gamma }{\delta }}\Gamma \left( \frac{\gamma }{\delta }+\frac{\beta }{%
1-\alpha }+1\right) }{\Gamma \left( \frac{\gamma }{\delta }\right) \Gamma
\left( \frac{\beta }{1-\alpha }+1\right) }, & \text{for } & \alpha <1 \\ 
\frac{1}{2}\frac{\delta \left[ a\left( 1-\alpha \right) \right] ^{\frac{
\gamma }{\delta }}\Gamma \left( \frac{\beta }{\alpha -1}\right) }{\Gamma
\left( \frac{\gamma }{\delta }\right) \Gamma \left( \frac{\beta }{\alpha -1}%
- \frac{\gamma }{\delta }\right) }, & \text{for } & \frac{1}{1-\alpha }-%
\frac{\gamma }{\delta }>0,\alpha >1 \\ 
\frac{1}{2}\frac{\left( a\beta \right) ^{\frac{\gamma }{\delta }}}{\Gamma
\left( \frac{\gamma }{\delta }\right) }, &  & \alpha \rightarrow 1%
\end{array}
\right.
\end{align*}

Note that for $\alpha <1$ it is a finite range density with $\left[
1-a\left( 1-\alpha \right) \left\vert x\right\vert ^{\delta }\right] >0$ and
\ ($\ref{eqn-path-2}$) remains in the extended generalized type-1 beta
family \ . The pathway density in ($\ref{eqn-path-2}$), for $\alpha < 1$,
includes the extended type-1 beta density, the triangular density, the
uniform density and many other p.d.f'.s .\cite{Praveen-pathway}.For instance
, $\alpha >1$ gives

\begin{equation}
f\left( x\right) =c\left\vert x\right\vert ^{\gamma -1}\left[ 1+a\left(
1-\alpha \right) \left\vert x\right\vert ^{\delta }\right] ^{-\frac{\beta }{%
\left( 1-\alpha \right) }}  \label{eqn-path-3}
\end{equation}%
provided that $-\infty <x<\infty ,\delta >0,\beta \geq 0,$ and $\alpha >0$
which is the extended generalized type-2 beta model for real x. It includes
the type-2 beta density, the F density, the Student-t density, the Cauchy
density and many more.The composition of the integral transform operator ($%
\ref{eqn-path-1}$) with the product of generalized Bessel function of the
first kind is given in \cite{Praveen-pathway}.

The purpose of this work is to investigate the composition formula of
integral transform operator due to Nair, which is expressed in terms of the
generalized Wright hypergeometric function, by inserting the generalized
Struve function of the first kind $S_{\alpha }\left( \lambda z\right) $
which is defined in the equation ($\ref{3}$) . The results given in this
section are based on the preliminary assertions giving by composition
formula of pathway fractional integral ($\ref{eqn-path-1}$) with a power
function.

\begin{lemma}
\label{lem-3} \cite{Nair-1} Let $\eta \in C$, $Re\left( \eta \right) >0,$ $%
\beta \in C$ and $\alpha <1. If $ $\RM\left( \beta \right) >0, $ and $\RM
\left( \frac{\eta }{1-\alpha }\right) >-1,$  then

\begin{equation}
\left\{ P_{0+}^{\left( \eta ,\alpha \right) }\left[ t^{\beta -1}\right]
\right\} \left( x\right) =\frac{x^{\eta +\beta }}{\left[ a\left( 1-\alpha
\right) \right] ^{\beta }}\frac{\Gamma \left( \beta \right) \Gamma \left( 1+%
\frac{\eta }{1-\alpha }\right) }{\Gamma \left( 1+\frac{\eta }{1-\alpha }%
+\beta \right) }  \label{eqn-path-2.1}
\end{equation}

The pathway fractional integration of the Bessel Struve Kernel function of
the first kind is given by the following result.
\end{lemma}

\begin{theorem}
Let $\eta ,\sigma ,p,b,c,\lambda $ $\in \mathbb{C}$ and $\alpha <1$ be such
that $Re\left( \eta \right) >0,Re\left( \sigma \right) >0,Re\left( \sigma
+n\right) >0$\text{ and }$Re\left( \frac{\eta }{1-\alpha }\right) >-1$ then
the following formula hold%
\begin{equation}
\begin{array}{c}
\left( P_{0+}^{\left( \eta ,\alpha \right) }\left[ t^{\sigma -1}S_{\nu
}\left( \lambda t\right) \right] \right) \left( x\right) =x^{\eta +\sigma }%
\frac{\Gamma \left( \nu +1\right) \mu \left( 1+\frac{\eta }{1-\alpha }%
\right) }{\sqrt{\pi }\left[ a\left( 1-\alpha \right) \right] ^{\sigma +p+1}}
\times _{2}\Psi _{2}\left[ 
\begin{array}{cc}
\left( \frac{1}{2},\frac{1}{2}\right) , & \left( \rho ,1\right) \\ 
\left( \nu +1,\frac{1}{2}\right) , & \left( 1+\frac{\eta }{1-\alpha }+\sigma
,1\right)%
\end{array}%
;\lambda x\right]%
\end{array}
\label{eqn-4}
\end{equation}

\begin{proof}
Applying ($\ref{3}$), and ($\ref{eqn-path-1}$) and changing the order of
integration and summation, we get%
\begin{eqnarray*}
\left( P_{0+}^{\left( \eta ,\alpha \right) }\left[ t^{\sigma -1}S_{\nu
}\left( \lambda t\right) \right] \right) \left( x\right) &=& \left(
P_{0+}^{\left( \eta ,\alpha \right) }\left[ t^{\sigma -1}\sum_{n=0}^{\infty }%
\frac{\lambda ^{n}\Gamma \left( \alpha +1\right) \Gamma \left( \left(
n+1\right) /2\right) }{\sqrt{\pi }n!\Gamma \left( n/2+\alpha +1\right) }t^{n}%
\right] \right) \left( x\right) \\
&=&\sum\limits_{k=0}^{\infty }\frac{\lambda ^{n}\Gamma \left( \nu +1\right)
\Gamma \left( \left( n+1\right) /2\right) }{\sqrt{\pi }n!\Gamma \left(
n/2+\nu +1\right) }\left( P_{0+}^{\left( \eta ,\alpha \right) }\left\{
t^{(n+\sigma )-1}\right\} \right) \left( x\right)
\end{eqnarray*}

Using the conditions mentioned in the statement of the theorem and $k\in
K_{0}$ ,$R\left( p+n\right) >0,Re\left( \frac{\eta }{1-\alpha }\right) >-1.$
Applying Lemma $\ref{lem-3}$ and using ($\ref{eqn-path-2.1}$) with $\beta $
replaced by $\sigma +n$, we get

\begin{eqnarray*}
\left( P_{0+}^{\left( \eta ,\alpha \right) }\left[ t^{\sigma -1}S_{\nu
}\left( \lambda t\right) \right] \right) \left(
x\right)&=&\sum\limits_{k=0}^{\infty }\frac{\lambda ^{n}\Gamma \left( \nu
+1\right) \Gamma \left( \left( n+1\right) /2\right) }{\sqrt{\pi }n!\Gamma
\left( n/2+\nu +1\right) }\frac{x^{\eta +\alpha }}{\left[ a\left( 1-\alpha
\right) ^{^{\sigma +n}}\right] } \\
&&\times \frac{\Gamma \left( \sigma +n\right) \Gamma \left( 1+\frac{\eta }{%
1-\alpha }\right) }{\Gamma \left( 1+\frac{\eta }{1-\alpha }+\sigma +n\right) 
} \\
&=&\frac{x^{\eta +\alpha }\Gamma \left( \nu +1\right) \Gamma \left( 1+\frac{%
\eta }{1-\alpha }\right) }{\sqrt{\pi }\left[ a\left( 1-\alpha \right)
^{\sigma }\right] } \\
&&\times \sum\limits_{k=0}^{\infty }\frac{\Gamma \left( \frac{n}{2}+\frac{1}{%
2}\right) \Gamma \left( \sigma +n\right) }{\Gamma \left( \frac{n}{2}+\nu
+1\right) \Gamma \left( 1+\frac{\eta }{1-\alpha }+\sigma +n\right) }\frac{%
\left( x\lambda \right) ^{n}}{n!}
\end{eqnarray*}

which gives the desired result
\end{proof}
\end{theorem}

By considering the relations given in \ref{e1} and \ref{e2}, we obtain
various new integral formulas for Bessel Struve functions involving in the
Pathway fractional integration Operators:

\begin{theorem}
Let $\eta ,\sigma ,p,b,c$ $\in C$ and $\alpha <1$ such that $Re\left( \eta
\right) >0,Re\left( \sigma +n\right) >0$ and $Re\left( \frac{\eta }{1-\alpha 
}\right) >-1$ then the following formula hold%
\begin{eqnarray*}
\left( P_{0+}^{\left( \eta ,\alpha \right) }\left[ t^{\sigma -1}e^{t}\right]
\right) \left( x\right) = x^{\eta +\sigma }\frac{\Gamma \left( 1+\frac{\eta 
}{1-\alpha }\right) }{\left[ a\left( 1-\alpha \right) \right] ^{\sigma }}
\times _{1}\Psi _{1}\left[ 
\begin{array}{c}
\left( \sigma ,1\right) ; \\ 
\left( 1+\sigma +\frac{\eta }{1-\alpha },1\right) ;%
\end{array}%
\frac{x}{\left[ a\left( 1-\alpha \right) \right] }\right]
\end{eqnarray*}

\begin{proof}
Applying ($\ref{e1}$),using ($\ref{eqn-path-1}$) \ with the help of Lemma $%
\ref{lem-3}$ and changing the order of integration and summation, we get%
\begin{eqnarray*}
\left( P_{0+}^{\left( \eta ,\alpha \right) }\left[ t^{\sigma -1}e^{t}\right]
\right) \left( x\right) &=&\left( P_{0+}^{\left( \eta ,\alpha \right) }\left[
t^{\sigma -1}S_{-\frac{1}{2}}\left( \lambda t\right) \right] \right) \left(
x\right) \\
&=&\left( P_{0+}^{\left( \eta ,\alpha \right) }\left[ t^{\sigma
-1}\sum_{n=0}^{\infty }\frac{\Gamma \left( \frac{1}{2}\right) \Gamma \left( 
\frac{n+1}{2}\right) }{\sqrt{\pi }n!\Gamma \left( \frac{n+1}{2}\right) }%
\right] \right) \left( x\right) \\
&=&\sum_{n=0}^{\infty }\frac{1}{n!}\left( P_{0+}^{\left( \eta ,\alpha
\right) }\left( t^{\sigma +n-1}\right) \right) \left( x\right) \\
&=&\sum_{n=0}^{\infty }\frac{1}{n!}\frac{x^{\eta +\sigma +n}\Gamma \left(
\sigma +n\right) \Gamma \left( 1+\frac{\eta }{1-\alpha }\right) }{\left[
a\left( 1-\alpha \right) \right] ^{\sigma +n}\Gamma \left( 1+\frac{\eta }{%
1-\alpha }+\sigma +n\right) } \\
&=&\frac{\Gamma \left( 1+\frac{\eta }{1-\alpha }\right) }{\left[ a\left(
1-\alpha \right) \right] ^{\sigma }}x^{\eta +\sigma } \sum_{n=0}^{\infty }%
\frac{\Gamma \left( \sigma +n\right) }{\left( a\left( 1-\alpha \right)
\right) ^{n}\Gamma \left( 1+\frac{\eta }{1-\alpha }+\sigma +n\right) }\frac{%
x^{n}}{n!}
\end{eqnarray*}

which completes the proof of the theorem.
\end{proof}
\end{theorem}

\begin{theorem}
\bigskip Let $\eta ,\sigma ,p,b,c$ $\in C$ and $\alpha <1$ such that $%
Re\left( \sigma +n\right) >0$ and $Re\left( \frac{\eta }{1-\alpha }\right)
>-1$ then the following formula hold%
\begin{eqnarray*}
\left( P_{0+}^{\left( \eta ,\alpha \right) }\left[ t^{\sigma -1}\frac{%
-1+e^{t}}{t}\right] \right) \left( x\right) = x^{\eta +\sigma }\frac{\Gamma
\left( 1+\frac{\eta }{1-\alpha }\right) }{2\left[ a\left( 1-\alpha \right) %
\right] ^{\sigma }} \times _{2}\Psi _{2}\left[ 
\begin{array}{c}
\left( \frac{1}{2},\frac{1}{2}\right) ,\left( \sigma ,1\right) ; \\ 
\left( \frac{1}{2},\frac{1}{2}\right) ,\left( 1+\sigma +\frac{\eta }{%
1-\alpha },1\right) ;%
\end{array}%
\frac{x\lambda }{\left[ a\left( 1-\alpha \right) \right] ^{{}}}\right]
\end{eqnarray*}

\begin{proof}
Applying ($\ref{e2}$),using ($\ref{3}$) , ($\ref{eqn-path-1}$) and Lemma $1$
and changing the order of integration and summation, we get%
\begin{eqnarray*}
\left( P_{0+}^{\left( \eta ,\alpha \right) }\left[ t^{\sigma -1}\frac{%
-1+e^{t}}{2}\right] \right) \left( x\right) &=&\left( P_{0+}^{\left( \eta
,\alpha \right) }\left[ t^{\sigma -1}S_{\frac{1}{2}}\left( \lambda t\right) %
\right] \right) \left( x\right) \\
&=&\left( P_{0+}^{\left( \eta ,\alpha \right) }\left[ t^{\sigma
-1}\sum_{n=0}^{\infty }\frac{\Gamma \left( \frac{3}{2}\right) \Gamma \left( 
\frac{n+1}{2}\right) }{\sqrt{\pi }n!\Gamma \left( \frac{n+3}{2}\right) }%
\right] \right) \left( x\right) \\
&=&\sum_{n=0}^{\infty }\frac{\Gamma \left( \frac{3}{2}\right) \Gamma \left( 
\frac{n+1}{2}\right) \lambda ^{n}}{\sqrt{\pi }n!\Gamma \left( \frac{n+3}{2}%
\right) }\left( P_{0+}^{\left( \eta ,\alpha \right) }\left( t^{\sigma
+n-1}\right) \right) \left( x\right) \\
&=&\frac{\Gamma \left( \frac{3}{2}\right) }{\sqrt{\pi }}\sum_{n=0}^{\infty }%
\frac{\Gamma \left( \frac{n+1}{2}\right) }{n!}\frac{x^{\eta +\sigma
+n}\Gamma \left( \sigma +n\right) \Gamma \left( 1+\frac{\eta }{1-\alpha }%
\right) }{\left[ a\left( 1-\alpha \right) \right] ^{\sigma +n}\Gamma \left(
1+\frac{\eta }{1-\alpha }+\sigma +n\right) } \\
&=&\frac{\Gamma \left( 1+\frac{\eta }{1-\alpha }\right) }{2\left[ a\left(
1-\alpha \right) \right] ^{\sigma }}x^{\eta +\sigma } \sum_{n=0}^{\infty }%
\frac{\Gamma \left( \frac{n+1}{2}\right) \Gamma \left( \sigma +n\right) }{%
\Gamma \left( \frac{n+3}{2}\right) \Gamma \left( 1+\frac{\eta }{1-\alpha }%
+\sigma +n\right) }\frac{\left( x\lambda \right) ^{n}}{n!\left( a\left(
1-\alpha \right) \right) ^{n}}
\end{eqnarray*}

Similarly one can derive the pathway integral representation of %
\eqref{eqn-r1}, \eqref{eqn-r2}.
\end{proof}
\end{theorem}

\textbf{Conclusion}

In this present study, we consider Bessel Struve kernel function $S_{\alpha
}\left( \lambda z\right) ,\lambda ,z\in \mathbb{C}$ to obtain the results in
terms of generalized Wright functions by applying Marichev-Saigo-Maeda
operators. Also, the pathway integral representations Bessel struve kernel
function and its relation between many other functions also derived in this
study.


\begin{thebibliography}{99}
\bibitem{Balenu-b1} D. Baleanu, About fractional quantization and fractional
variational principles. Communication in Nonlinear Sciences and Numerical
Simulation. 2009; 14(6):2520--2523.

\bibitem{Kiriyakova-b1} V. Kiryakova, Generalized Fractional Calculus and
Applications. Vol. 301. Essex, UK: Longman Scientific \& Technical; 1994.

\bibitem{Balenu-b2} D. Baleanu , OG. Mustafa, On the global existence of
solutions to a class of fractional differential equations. Computers \&
Mathematics with Applications. 2010;59(5):1835--1841

\bibitem{Purohit-b1} Purohit SD, Kalla SL. On fractional partial
differential equations related to quantum mechanics. Journal of Physics A.
2011;44(4) 045202

\bibitem{Love} Love ER. Some integral equations involving hypergeometric
functions. Proceedings of the Edinburgh Mathematical Society.
1967;15(3):169--198.

\bibitem{McBride} A. C. McBride, Fractional powers of a class of ordinary
differential operators, Proc. London Math. Soc. (III) 45 (1982), 519--546.

\bibitem{Kalla1} S. L. Kalla, Integral operators involving Fox's H-function
I, Acta Mexicana Cienc. Tecn. 3 (1969), 117--122.

\bibitem{Kalla2} S. L. Kalla, Integral operators involving Fox's H-function
II, Acta Mexicana Cienc. Tecn. 7 (1969), 72--79.

\bibitem{Kalla3} S. L. Kalla - R. K. Saxena, Integral operators involving
hypergeometric functions, Math. Z. 108 (1969), 231--234.

\bibitem{Kalla4} S. L. Kalla - R. K. Saxena, Integral operators involving
hypergeometric functions II, Univ. Nac. Tucum%
%TCIMACRO{\U{b4}}%
%BeginExpansion
\'{}%
%EndExpansion
an, Rev. Ser., A24 (1974), 31--36.

\bibitem{Saigo1} M. Saigo, A remark on integral operators involving the
Gauss hypergeometric functions, Math. Rep. Kyushu Univ. 11 (1978), 135--143.

\bibitem{Saigo2} M. Saigo, A certain boundary value problem for the
Euler-Darboux equation I, Math. Japonica 24 (4) (1979), 377--385.

\bibitem{Saigo3} M. Saigo, A certain boundary value problem for the
Euler-Darboux equation II, Math. Japonica 25 (2) (1980), 211--220.

\bibitem{Saigo4} M. Saigo\ and\ N. Maeda, More generalization of fractional
calculus, in \textit{Transform methods \& special functions, Varna '96},
386--400, Bulgarian Acad. Sci., Sofia.

\bibitem{Srivastava-Karlson} H. M. Srivastava - P.W. Karlson, Multiple
Gaussion Hypergeometric Series, Ellis Horwood Limited, New York, 1985.

\bibitem{Erdelyi} A. Erd%
%TCIMACRO{\U{b4}}%
%BeginExpansion
\'{}%
%EndExpansion
elyi -W. Magnus - F. Oberhettinger - F. G. Tricomi, Higher Transcendental
Functions, Vol.1, McGraw-Hill, New York, 1953.

\bibitem{Prudnikov} A. P. Prudnikov - Yu. A. Brychkov - O. I. Marichev,
Integrals and Series. Special Functions, Vol. 1-5, Gordon \& Breach, New
York, 1986.

\bibitem{Kim-Lee-HMSri} Y. C. Kim, K. S. Lee\ and\ H. M. Srivastava, Some
applications of fractional integral operators and Ruscheweyh derivatives, J.
Math. Anal. Appl. \textbf{197} (1996), no.~2, 505--517.

\bibitem{Kir1} V. Kiryakova, All the special functions are fractional
differintegrals of elementary functions, J. Phys. A \textbf{30} (1997),
no.~14, 5085--5103.

\bibitem{Miller-Ross} K. S. Miller - B. Ross, An Introduction to the
Fractional Calculus and Differential Equations, A Wiley Interscience
Publication, John Wiley and Sons Inc., New York, 1993.

\bibitem{Kiriyakova-b2} V. Kiryakova, Generalized Fractional Calculus and
Applications, Longman Scientific \& Tech., Essex, 1994.

\bibitem{Marichev1} O. I. Marichev, Volterra equation of Mellin convolution
type with a Horn function in the kernel (In Russian). Izv. AN BSSR Ser.
Fiz.-Mat. Nauk 1 (1974), 128--129.

\bibitem{saxena-saigo} R. K. Saxena\ and\ M. Saigo, Generalized fractional
calculus of the $H$-function associated with the Appell function $F\sb 3$,
J. Fract. Calc. \textbf{19} (2001), 89--104.

\bibitem{HMSri-frac-Hfunction} H. M. Srivastava,\ S.-D. Lin\ and\ P.-Y.
Wang, Some fractional-calculus results for the $\overline{H}$-function
associated with a class of Feynman integrals, Russ. J. Math. Phys. \textbf{%
13.}

\bibitem{Samko1} S. Samko - A. Kilbas - O. Marichev, Fractional Integrals
and Derivatives. Theory and Applications, Gordon \& Breach Sci. Publ., New
York, 1993.

\bibitem{Fox} C. Fox, The asymptotic expansion of generalized hypergeometric
functions, Proc. London. Math. Soc. \textbf{27} (1928), no.~4, 389-400.

\bibitem{Wright-2} E. M. Wright, The asymptotic expansion of integral
functions defined by Taylor series, Philos. Trans. Roy. Soc. London, Ser. A. 
\textbf{238} (1940), 423--451.

\bibitem{Wright-3} E. M. Wright, The asymptotic expansion of the generalized
hypergeometric function, Proc. London Math. Soc. (2) \textbf{46} (1940),
389--408.

\bibitem{Kilbas} A. A. Kilbas, M. Saigo\ and\ J. J. Trujillo, On the
generalized Wright function, Fract. Calc. Appl. Anal. \textbf{5} (2002),
no.~4, 437--460.

\bibitem{Kilbas-itsf} A. A. Kilbas\ and\ N. Sebastian, Generalized
fractional integration of Bessel function of the first kind, Integral
Transforms Spec. Funct. \textbf{19} (2008), no.~11-12, 869--883.

\bibitem{Kilbas-frac} A. A. Kilbas\ and\ N. Sebastian, Fractional
integration of the product of Bessel function of the first kind, Fract.
Calc. Appl. Anal. \textbf{13} (2010), no.~2, 159--175.

\bibitem{HMSri-Wright} H. M. Srivastava, Some Fox-Wright generalized
hypergeometric functions and associated families of convolution operators,
Appl. Anal. Discrete Math. \textbf{1} (2007), no.~1, 56--71.

\bibitem{Watson} G.N. Watson, A treatise on the theory of Bessel functions,
Cambridge university press (Cambridge, 1992)

\bibitem{Gasmi} A. Gasmi and M. Sifi, The Bessel-Struve intertwinning
operator on C and mean-periodic functions, IJMMS 2004:59, 3171--3185

\bibitem{Nair-1} S. S. Nair: Pathway fractional integration operator, Fract.
Calc. Appl. Anal., 12(3) (2009), 237-252.

\bibitem{Mathai-pathway} A. M. Mathai : A pathway to matrix-variate gamma
and normal densities, Linear Algebra Appl., 396 (2005), 317-328.

\bibitem{Mathai-Habold-1} A. M. Mathai, H. J. Haubold: On generalized
distributions and path-ways, Phys. Lett. A, 372 (2008), 2109-2113.

\bibitem{Mathai-Habold-2} A. M. Mathai, H. J. Haubold: Pathway model,
superstatistics, Tsallis statistics and a generalize measure of entropy,
Phys. A, 375 (2007), 110-122.

\bibitem{Praveen-pathway} Dumitru Baleanu, Praveen Agarwal: A Composition
Formula of the Pathway Integral Transform Operator,Note di Matematica,Note
Mat. 34 (2014) no. 2, 145--155.

\bibitem{Rainville} E. D. Rainville, \textit{Special functions}, Macmillan,
New York, 1960.

\bibitem{M Saigo} M. Saigo, A remark on integral operators involving the
Gauss hypergeometric functions, Math. Rep. Kyushu Univ. \textbf{11}
(1977/78), no.~2, 135--143.

\bibitem{Nisar} S.R Mondal, K.S Nisar, Marichev-Saigo-Maeda Fractional
Integration Operators Involving Generalized Bessel Functions, Mathematical
Problems in Engineering, Volume 2014, Article ID 274093, 11 pages
\end{thebibliography}
\end{document}